\documentclass[12pt, reqno]{amsart}
\usepackage{amsmath}
\usepackage{amsfonts}
\usepackage{amssymb}
\usepackage[english]{babel}
\usepackage{color}

\usepackage{mathtools}
\mathtoolsset{showonlyrefs}

\usepackage{hyperref}
\usepackage{comment}
\usepackage{calrsfs}

\usepackage{multicol}
\usepackage{graphicx}

\usepackage[normalem]{ulem}

\usepackage{wrapfig}

\theoremstyle{plain}
\newtheorem {theorem}{Theorem}[section]
\newtheorem {lemma}[theorem]{Lemma}

\theoremstyle{definition}
\newtheorem{definition}[theorem]{Definition}
\newtheorem{rmk}[theorem]{Remark}

\theoremstyle{remark}

\numberwithin{equation}{section}

\newcommand{\N}{\mathbb{N}}
\newcommand{\R}{\mathbb{R}}

\renewcommand{\P}{P_\alpha}

\newcommand{\cL}{\mathcal{L}}

\newcommand{\X} {\mathcal{X}\! f}

\newcommand{\e}{\varepsilon}

\newcommand{\s}{\sigma}
\renewcommand{\d}{\partial}
\newcommand{\de}{\delta}
\renewcommand{\a}{\alpha}
\newcommand{\y}{{x_{n+1}}}
\newcommand{\x}{{\bar x}}

\renewcommand{\phi}{\varphi}
\newcommand{\De}{\Delta}
\renewcommand{\S}{\Sigma}
\newcommand{\<}{\langle}
\renewcommand{\>}{\rangle}

\newcommand{\res}
{\mathop{\hbox{\vrule height 7pt width .5pt depth 0pt \vrule
height .5pt width 6pt depth 0pt}}\nolimits}


\begin{document}
\author[V.~Franceschi]{Valentina Franceschi$^*$}

\author[R.~Monti]{Roberto Monti$^*$}

\author[A.~Socionovo]{Alessandro Socionovo$^\dagger$}
 
\address{$^*$Dipartimento di Matematica Tullio Levi Civita, Universit\`a di Padova, Italy}
\email{valentina.franceschi@unipd.it}
\email{monti@math.unipd.it}

\address{$^\dagger$Laboratoire Jacques-Louis Lions, CNRS, Inria, Sorbonne Université, Université de Paris, France}
\email{alesocio1994@gmail.com}

\title{Mean value formulas on surfaces in Grushin spaces}

\begin{abstract} We prove (sub)mean value formulas at the point $0\in\Sigma$  for (sub)harmo\-nic functions on a hypersurface $\Sigma\subset\R^{n+1}$ where the differentiable structure and the surface measure depend on the ambient Grushin structure.
\end{abstract}

\maketitle


\renewcommand{\x}{x}
\renewcommand{\y}{y}

\renewcommand{\rho}{\varrho}

\renewcommand{\X}{X^2}

\newcommand{\dil}{d}
\renewcommand{\r}{\varrho}

\section*{Acknowledgments}

\begin{wrapfigure}{r}{0.1\linewidth}
	\vspace{-.25cm}
	\hspace{-.35cm}
	\includegraphics[width=1.2\linewidth]{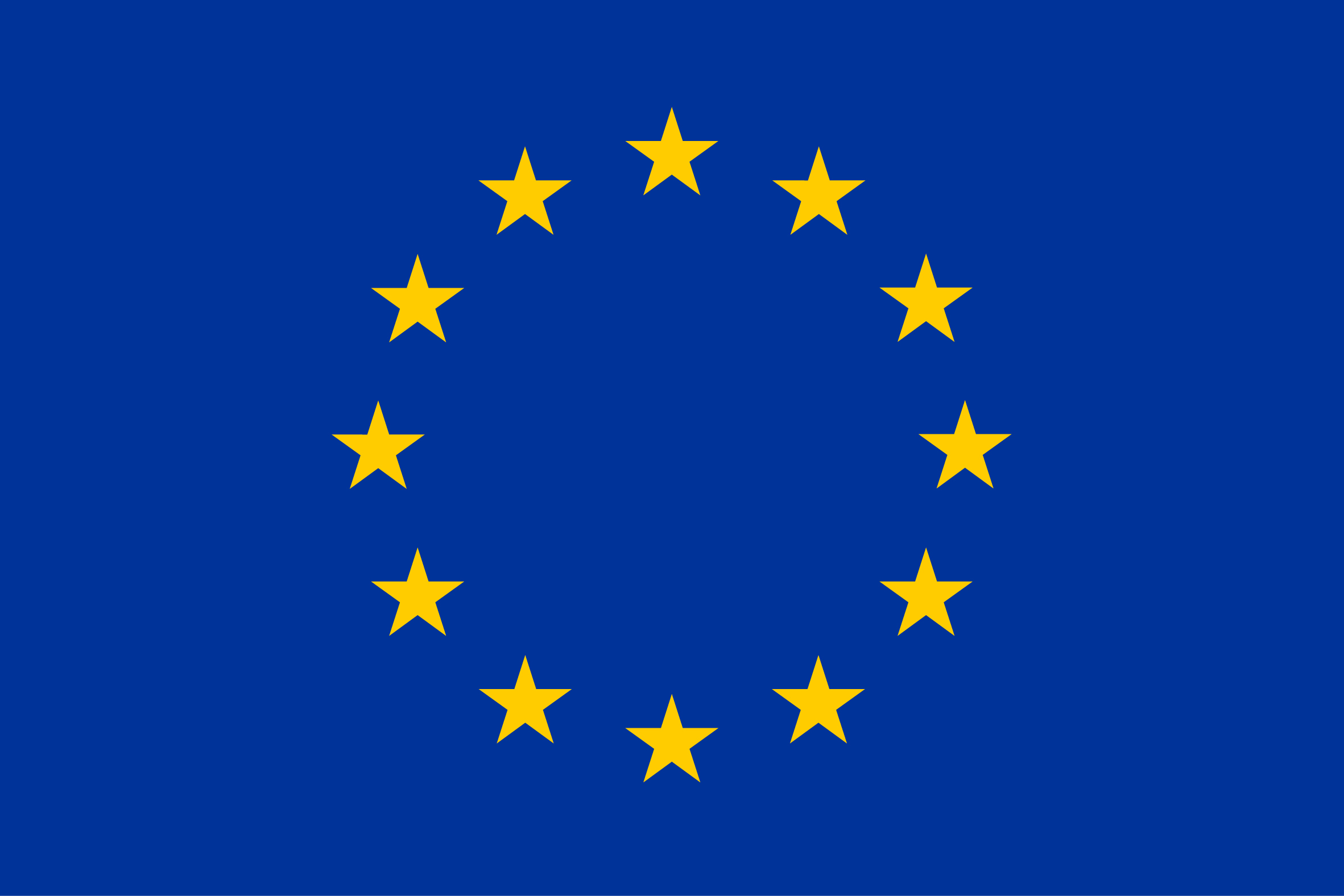}
	\vspace{-.2cm}
\end{wrapfigure}
This project has received funding from the European Union’s Horizon 2020 research and innovation programme under the Marie Skłodowska-Curie grant agreement No 101034255.

\section{Introduction}

For $n\in\mathbb N $ and $\alpha> 0$, we consider the vector fields on $\R^{n+1}$
\begin{equation}\label{eq:vector_fields}
\begin{array}{l}
\displaystyle X_i =\frac{\d}{\d {x_i} },
\quad i=1,\ldots,n,\quad 
\displaystyle X_{n+1} =|\x|^\alpha \frac{\d}{\d \y }.
\end{array}
\end{equation}
Here, a generic point in $\R^{n+1}$ is denoted by $\xi=(\x ,\y)=(x_1,\ldots, x_n, \y) \in \R^{n+1}$. We also consider the second order partial differential operator on $\R^{n+1}$ given by
\begin{equation}
\label{eq:cL}
  \cL \varphi  = \sum_{i=1}^{n+1} X_i^2\varphi   = \Delta_{\x}\varphi  +|\x|^{2\a} \partial^2_{\y}\varphi,
\end{equation}
where $|\x| = (x_1^2+\ldots+x_n^2)^{1/2}$. 
The operator $\cL$ in \eqref{eq:cL}   is known as Baouendi-Grushin operator, see \cite{G19} for a historical account and see also \cite{FL83}.

  When $\a$ is an even integer, this operator is hypoelliptic and 
 admits a   fundamental solution with pole at  any point $\xi_0\in\R^{n+1}$ (see~\cite{BFI15} for an explicit representation). When $\xi_0=0$,
 an explicit formula for this fundamental solution is in fact known for any $\a>0$ (see \cite{G})  and, up to a normalization constant, it is the function $\Gamma (\xi)=\rho(\xi)^{1-n-\alpha} $, $\xi\neq0$,  where  $\rho:\R^{n+1}\to\R$ is the gauge function 
\begin{equation}
  \label{def:Gauge}
  \rho(\xi)=\big (|\x|^{2(\a+1)}+(\a+1)^2\y^2\big ) ^{\frac{1}{2(\a+1)}}.
\end{equation}

Let $\Sigma\subset\R^{n+1}$ be a hypersurface of class $C^2$ with $0\in\Sigma$. 
We declare the  vector-fields \eqref{eq:vector_fields}
orthonormal and we project them onto the tangent space to $\Sigma$, getting tangential operators $\delta_1,\ldots,\delta_{n+1}$.
We fix on $\Sigma$ the hypersurface measure $\sigma$ associated with \eqref{eq:vector_fields} according to the theory of sub-elliptic perimeters and then we define the adjoint operators 
$\delta_1^*,\ldots,\delta_{n+1}^*$ with respect to $\sigma$. The natural restriction of $\cL$ to $\Sigma$ is the differential operator
\[
 \cL_\Sigma =-\sum_{i=1}^{n+1}\de_i^*\de_i.
\] 
In this paper, we investigate the validity of mean value formulas (sub-mean value formulas) at the point $0\in \S$ for functions $f\in C^2(\Sigma)$ satisfying
$\cL _\Sigma f=0$ ($\cL_\Sigma f\geq 0$, respectively).

The operator $\cL$ is an example of ``sub-Laplacian" or ``sum of squares of vector fields"
satisfying the H\"ormander condition \cite{H67}.   When $\cL=\sum_{i=1}^m X_i^2$ is such an operator in $\R^{n+1}$, the validity of mean-value formulas for $\cL$-harmonic functions 
is established in~\cite{CGL93,BL13,CL21}. Denoting
by
$\Gamma(\cdot,\xi_0)$ the  fundamental solution for $\cL$ with pole at $\xi_0$,
if a function $f$ satisfies $\cL f=0$ then for any $r>0$ and $\xi_0\in\R^{n+1}$
\begin{equation}
\label{eq:mvf}
  f(\xi_0)=\frac{1}{r}\int_{\Omega_r(\xi_0)}f(\xi) K(\xi,\xi_0) \; d\xi,
\end{equation}
where  $
\Omega_r(\xi_0)= \{\xi\in\R^{n+1}: \Gamma(\xi,\xi_0) >1/r\}$ and 
$
 K(\xi,\xi_0) = {|X\Gamma(\xi,\xi_0) |^2}/{\Gamma(\xi,\xi_0) ^2}$, with $|X\Gamma(\cdot, \xi_0)|^2=\sum_{i=1}^{m}(X_i\Gamma(\cdot,\xi_0))^2$.  
The appearance of the kernel $K$  is due to the different symmetry of Carnot-Carath\'eodory balls   associated with the vector-fields building up $\cL$ and level sets of $\Gamma(\cdot,\xi_0)$.

In the Riemannian case, the validity of   mean-value formulas on metric balls for harmonic functions leads to the notion of ``harmonic manifold".
Starting probably with~\cite{W50}, there exists a huge literature on the problem of characterizing harmonic manifolds and it is not possible to give a full account, here. In fact, our hypersurface $\Sigma$ embedded in $\R^{n+1}$ with the Grushin structure is not a Riemannian manifold but rather a weighted 
Riemannian manifold that becomes singular at the point $0\in\S$, see Remark \ref{REM}.  

In the Grushin space, the harmonicity at $0\in \S$ is governed by the following structural function $q_\S:\Sigma\setminus\{0\}\to \R$
 \begin{equation}
  \label{eq:exactq}
  q_\S(\xi) =\<X\rho,\nu\>\Big[(n+3\a)\<X\log\rho,\nu\> -2\a\<\nabla_\x\log|\x|,\bar\nu\> +nH_\S\Big].
 \end{equation} 
Above, $X\rho =(X_1\rho,\ldots,X_{n+1}\rho)$ is the $X$-gradient of the gauge function $\rho$, $\nu= (\bar \nu,\nu_{n+1})$ is the $\a$-normal to $\Sigma$, $\langle \cdot,\cdot\rangle$ are standard scalar products in $\R^{n+1}$ and $\R^n$, and $H_\S$ is the mean curvature of $\Sigma$ associated with the Grushin structure.
We say that $\Sigma$ is $\a$-harmonic if $q_\S=0$. In particular, any homogeneous hypersurface, $\<X\rho,\nu\>=0$, is $\a$-harmonic, as we show in Section~\ref{ss:homo}.
\begin{theorem}\label{MVF}
 Let $\S\subset\R^{n+1}$ be an $\a$-harmonic  hypersurface of class $C^2$ with $0\in\S$.
Any function $f \in C^2(\S)$ such that $   \cL_\Sigma f=0$
  satisfies the mean-value formula at $0$
 \begin{equation}
  \label{mvf}
  f(0) =\frac{C_{\S,n,\a} }{r^{n+\a} } \int_{B_r\cap \S }f(\xi) \, |\delta\rho(\xi)|^2  d\sigma,
 \end{equation}
 for all $r\in(0,r_0)$ and for some $r_0>0$ depending on $\S$.
The constant
 $0<C_{\S,n,\a}<\infty$ is defined by 
 \begin{equation}
  \label{mvfC}
\frac{1}{C_{\S,n,\a} } =\frac{1}{r^{n+\a} } \int_{B_r\cap \S }    |\delta\rho(\xi)|^2  d\sigma,
 \end{equation}
 where the right hand-side does not depend on $r\in (0,r_0)$.
 \end{theorem}

Above, the balls are 
\begin{equation}
\label{eq:Br}
B_r = \{\xi\in\R^{n+1}: \rho(\xi)<r\}.
\end{equation} 
and  $|\delta\rho|\leq 1$ is the length of the tangential gradient of $\rho$. When $\S$ is homogeneous, the kernel is $|\delta\rho|^2 = |\x|^{2\a}/\rho^{2\a}$.
In the case of $\a$-subharmonic hypersurfaces, $q_\S\geq 0$, the statement is similar.

\begin{theorem}\label{thm:SMVF} Let $\S\subset\R^{n+1}$ be an $\a$-subharmonic  hypersurface of class $C^2$ with $0\in\S$.
Any function $f \in C^2(\S)$ such that $   \cL_\Sigma f\geq 0$
  satisfies the following sub-mean-value formula at $0$
 \begin{equation}
  \label{smvf}
  f(0) \leq \frac{C_{\S,n,\a} }{r^{n+\a} } \int_{B_r\cap \S }f(\xi) \, |\delta\rho(\xi)|^2  d\sigma,
 \end{equation}
 for all $r\in(0,r_0)$ and for some $r_0>0$ depending on $\S$. The constant
 $0<C_{\S,n,\a}<\infty$ is defined by 
 \begin{equation}
  \label{mvfC2}
\frac{1}{C_{\S,n,\a} } =\lim_{r\to0^+} \frac{1}{r^{n+\a} } \int_{B_r\cap \S }    |\delta\rho(\xi)|^2  d\sigma.
 \end{equation}
 \end{theorem}
{ 
The operator $\mathcal L$ in \eqref{eq:cL} and the hyper-surface measure $\sigma$ are invariant with respect to the vertical translations $(x,y)\mapsto (x,y+y_0)$, for any fixed $y_0\in\R$. Theorem \ref{thm:SMVF} can be therefore extended to get mean value formulas at points $(0,y_0)\in \Sigma$ also with $y_0\neq 0$.
Obtaining mean value formulas at points $(x_0,y_0)\in \Sigma$ with $x_0 \neq 0$ is, instead, difficult because  our knowledge of the fundamental solution of $\mathcal L$ with pole at $(x_0,y_0)$ with $x_0\neq 0$ is not explicit enough. }

Our interest in sub-mean value formulas on hypersurfaces of $\R^{n+1}$ endowed with a system of H\"ormander vector fields comes from the theory of minimal surfaces. 
One of the key tools in 
 Bombieri-De Giorgi-Miranda's proof
 of the gradient estimate is the sub-mean value property for sub-harmonic functions on minimal surfaces of the Euclidean space, see~\cite{BDeGM69}.  
   In our setting, a minimal surface is defined by $H_\S=0$. This condition simplifies the structural function $q_\S$, however, this is not sufficient to have $q_\S\geq0$.  
 
 The paper is organized as follows. In Section \ref{S2}, we recall the basic definitions of the measure $\s$, of the $\a$-normal $\nu$ of $\S$, and of mean curvature $H_\S$.
 In Section \ref{S3}, we introduce the various differential operators and we develop a calculus on radial functions.
 The explicit computations of second order derivatives of $\rho$ is crucial, here. In Section \ref{S4}, we prove Theorems \ref{MVF} and \ref{thm:SMVF}.
 Finally, in Section \ref{S5} we study the structural function $q_\S$.

\section{Perimeter and mean curvature of hypersurfaces}
\label{S2}

The \emph{$\a$-perimeter} of a Lebesgue measurable set 
$E\subset\R^{n+1}$ in an open set $A\subset\R^{n+1} $ is
\begin{equation}\label{eq:a_perimeter}
 \P(E;A) =   \sup \left\{ \int_{E} \sum_{i=1}^{n+1} X_i \varphi_i(\xi)\,
d\xi  \, :\, \varphi\in C_c^1(A;\R^{n+1}),\, 
\displaystyle  \max_{\xi \in A } |\varphi(\xi )|  \leq 1
 \right\}.
\end{equation}
We are using the Lebesgue measure $d\xi = d \mathcal L^{n+1} $  in $\R^{n+1}$.
When the boundary of $E$ is locally the graph of a Lipschitz function, its
 $\a$-perimeter 
 has the following integral representation
 (see~\cite[Proposition~2.1]{FM16})
\begin{equation}\label{eq:repr_formula}
    \P(E;A)=\int_{\d E\cap A}\sqrt{|\bar N|^2+|\x|^{2\a}  |N'|^2}\, d\mathcal H^{n}, 
    \end{equation}
where $N(\xi) =(\bar N (\xi)  ,N'(\xi) )\in\R^n\times\R$ is the Euclidean outer unit normal to $\d E$  at the point $\xi$, and 
$
\mathcal H^{n}$ is the standard $n$-dimensional Hausdorff measure in $\R^{n+1}$. On top of its appearance as a sub-Riemannian perimeter, the relevance of the perimeter $\P$ is due to its relation with the Heisenberg perimeter.
When  $\alpha=1$ and $n$ is even, then the Heisenberg perimeter of a set with cylindrical symmetry coincides 
with its $\a$-perimeter, see e.g., \cite[Proposition~2.3]{FM16}.

Motivated by \eqref{eq:repr_formula}, 
when  $\Sigma\subset\R^{n+1}$ is an orientable  hypersurface that is locally a Lipschitz graph and $N=(\bar N,N')$ is
its Euclidean normal, we call the Borel measure on $\Sigma$
\begin{equation}
 \label{eq:aPer_measure}
  \sigma = \sqrt{|\bar N|^2+|\x|^{2\a}  |N'|^2}\,  \mathcal H^{n} \res\S.
\end{equation}
the $\a$-perimeter measure of $\Sigma$.

The regular part of $\S$ is the set $\S^*=\{\xi=(\x,\y) \in\S:\x\neq 0\}$.
At  $\mathcal H^{n} $-a.e.~point $\xi \in\Sigma^*$  we can  define the $\a$-{\emph{normal}} of $\Sigma$ as the vector field
$\nu=\sum_{i=1}^{n+1} \nu_i X_i$ with 
\begin{equation}
\label{eq:nu}
\begin{split}
&     \nu_i =\frac{N_i}{\sqrt{|\bar N|^2+|x|^{2\alpha}  |N'|^2} },\quad \text{for }i=1,\dots,n,\\
  &   \nu_{n+1}   =\frac{a  N_{n+1 }}{\sqrt{|\bar N|^2+|x|^{2\alpha}  |N'|^2} }.    
\end{split}
\end{equation}
With  abuse of notation, we identify $\nu$ with the mapping $\nu:\S^*\to\R^{n+1}$ given by the   vector of its coordinates $\nu (\xi) =(\nu_1(\xi),\ldots,\nu_{n+1}(\xi ))\in\R^{n+1}$ for $\xi\in\Sigma^*$.

    \begin{rmk}[Comparison with the Riemannian structure]
\label{REM}
The hypersurface measure  $\sigma$ and  the $\a$-normal $\nu$ can be interpreted in the following Riemannian terms.
The tensor metric in $\R^{n+1}\setminus\{\x=0\}$  making $X_1,\dots,X_{n+1}$ orthonormal is 
\begin{equation}
\label{eq:g}
g_\a(\xi)=\begin{pmatrix}
I_n&0\\
0&|\x|^{-2\a}
\end{pmatrix},\quad \x\neq  0.
\end{equation}
When $\x=0$ the metric is not defined.
The Riemannian volume associated with $g_\a$  is the measure $\mu=|\x|^{-\alpha} \mathcal L^{n+1}  $ and is singular at $\x=0$.
The Riemannian surface area  associated with $g_\a$ of a hypersurface $\Sigma$  is the measure  
\[
  \mu_\S = |\x|^{-\a}  \sqrt{|\bar N|^2 + |\x|^{2\a} |N'|^2}    \mathcal H^{n}\res\S ,
\]
where $N = (\bar N,N')$ is the Euclidean unit normal.  We deduce that 
Lebesgue measure and $\a$-perimeter are  weighted 
Riemannian volume and hypersurface measures with the same weight:
\[
    \mathcal L^{n+1} = |\x| ^\alpha  \mu\quad\textrm{and}\quad
     \sigma =  |\x| ^\alpha  \mu_\Sigma.
\] 
\end{rmk}

We now focus on the case of graphs. Let  
 $\S=\Sigma_u=\{\xi=(\x,u(\x))\in \R^{n+1} : \x\in\Omega\}$ be the $y$-graph of a function  $u\in C^1(\Omega)$ for some open set $\Omega\subset\R^n$.
 We shall assume that $0\in\Omega$ and let $\Omega^* = \Omega\setminus\{0\}$.
 The  $\a$-{\emph{unit normal}} to $\Sigma_u $ at   points  in $\Sigma_u^*$  is  the mapping $\nu=(\bar\nu,\nu_{n+1}):\Sigma^*\to\R^{n+1} $
\begin{equation}
    \label{eq:nu_graph}
    \bar\nu
    =\frac{-\nabla u}{\sqrt{|\nabla u|^2+|\x|^{2\a} }},\quad 
    \nu_{n+1}=\frac{|\x|^{\a} }{{\sqrt{|\nabla u|^2+|\x|^{2\a} }} }.    
\end{equation}
This normal is pointing upwards.
Notice that $\nu=\nu(\xi)$ only depends on $\x$ and not on $\y=u(\x)$.

  From \eqref{eq:aPer_measure} and from the area-formula, we deduce that 
  the $\s$-area of $\S$ has the integral representation
 \begin{equation}
\label{eq:a_perimeter_graph}
\s(\S_u) = \int_{\Omega}\sqrt{|\nabla u|^2+|\x|^{2\a}}\;d\x= \int_{\Omega}v (\x) \;d\x,
\end{equation}
where $v$ is the $\s$-area element
\begin{equation}
    \label{eq:per_measure}
    { 
    v(\x)=\sqrt{|\nabla u|^2+ a ^{2}},\qquad
    a= |\x| ^{\a}.
    }
\end{equation}

If $\S_u$ minimizes the $\s$-area with respect to compact perturbations in $\Omega$ and $u\in C^2(\Omega)$, then   $u$
satisfies the partial differential equation of the minimal surface-type
\begin{equation}\label{EQ}
\operatorname{div}\left(\frac{\nabla u(\x) }{\sqrt{|\nabla u(\x) |^2+|\x| ^{2\a} }}\right) =0, \quad \x\in  \Omega^*.
\end{equation}
This 
follows by a standard first variation procedure applied to \eqref{eq:a_perimeter_graph}.
This suggests the following definition.

\begin{definition}[$\a$-mean curvature]
Let $\Sigma$ be the $\y$-graph of a function   $u\in C^2(\Omega)$.
We define the \emph{$\a$-mean curvature} of $\Sigma$  at the point $\xi = (\x, u(\x))\in\S^* $  as 
\begin{equation}
\label{eq:H}
    H_\S (\xi ) =\frac{1}{n}\operatorname{div}\left(\frac{\nabla u(\x) }{\sqrt{|\nabla u(\x) |^2+|\x|^{2\a} }}\right).
\end{equation}
We say that $\S$ is an $\a$-minimal hypersurface if $H_\S=0$ on $\Sigma^*$.
\end{definition}
A more geometric definition of $\a$-mean curvature will be presented in the next section.

\section{Tangential operators and Laplacians}
\label{S3}

We introduce   tangential differential operators on  hypersurfaces  in $\R^{n+1}$ endowed with the Grushin structure.
Let $\Sigma\subset\R^{n+1}$ be an embedded hypersurface of class $C^2$ with $\a$-normal  $\nu: \Sigma^* \to\R^{n+1}$.

The $X$-gradient of a function $\phi\in C^1(\R^{n+1})$ is the vector-field $X\phi =\sum_{i=1}
^{n+1} X_i\phi \, X_i$ that we identify, with abuse of notation, with the vector of its coordinates $X\phi = (X_1\phi,\ldots,X_{n+1}\phi)$.
We denote the standard scalar product on $\R^{n+1} $ by $\langle \cdot,\cdot \rangle $.

\begin{definition}[Tangential gradient]
Let $\nu:\Sigma^*\to\R^{n+1}$ be the $\a$-normal of $\Sigma$. 
The tangential gradient on $\Sigma$ is the mapping
$\delta : C^1(\Sigma)\to C(\Sigma^*;\R^{n+1})$ 
\begin{equation}
    \label{eq:delta}
    \de\varphi=X\varphi-\langle X\varphi,\nu\rangle\nu.
\end{equation}
For any $i=1,\ldots,n+1$ we also let $\de_i \varphi = X_i\varphi - \langle X\varphi,\nu\rangle\nu_i$.
\end{definition}

In \eqref{eq:delta}, $\varphi$ is extended outside $\Sigma$ in a $C^1$ way, and the definition will be independent of  this extension.
We are assuming that  $\Sigma$ is oriented and we are fixing a choice of $\a$-normal. The definition does not depend on this choice.  When $\Sigma$ is 
a $y$-graph, we agree that 
$\nu$ is pointing upwards.
In this case,  the $\a$-normal $\nu$ can be extended outside $\Sigma$ in a way that is independent of the variable $\y$. 
In the rest of the paper the surface $\Sigma$ will be always assumed to be a $\y$-graph. 

{ 

The definition of the tangential operator $\de$ in \eqref{eq:delta} is extrinsic. A different possibility could be to define the tangential gradient of functions on $\S$ using
 the Riemannian metric $g_\a$. However,  the choice in \eqref{eq:delta} is the correct one in order to recover the definition in  \eqref{eq:H} of $\a$-mean curvature.
 Indeed, this definition  reads
\begin{equation} \label{sic}
H_\S=-\frac1n \sum_{i=1}^nX_i\nu_i,
\end{equation}
and it can be rephrased in the following way.

}

\begin{lemma}\label{lem:relation_H_nu} Let $\Sigma\subset\R^{n+1}$ be a $\y$-graph of class $C^2$. Then on $\Sigma^*$ we have the identity
\[
H_\S=-\frac 1 n \sum_{i=1}^{n+1}\de_i\nu_i.
\]
\end{lemma}
\begin{proof} 
We   use the definition of $\de$ and observe that $\sum_{i=1}^{n+1}\nu_iX_k\nu_i=X_k(|\nu|^2/2)=0$, and $X_{n+1}\nu_{n+1}=0$, so that  
\[
\begin{split}
    \sum_{i=1}^{n+1}\de_i\nu_i&=
    \sum_{i=1}^{n}\de_i\nu_i+\de_{n+1}\nu_{n+1}\\
    &=\sum_{i=1}^{n}X_i\nu_i
    -\sum_{i=1}^{n}\sum_{k=1}^{n+1}\nu_i\nu_kX_k\nu_i
    +X_{n+1}\nu_{n+1}
    -\nu_{n+1}\nu_kX_k\nu_{n+1}\\
    &=\sum_{i=1}^{n}X_i\nu_i
    -\sum_{k=1}^{n+1}\left(\nu_k\sum_{i=1}^{n+1}\nu_iX_k\nu_i\right)
    +X_{n+1}\nu_{n+1}\\
    &=\sum_{i=1}^{n}X_i\nu_i=-nH_\S .
    \qedhere
\end{split}
\]
\end{proof}

Next we introduce the adjoint operators $\delta_i^*$, integrating by parts with respect to the  measure $\sigma$.

%

\begin{definition}[Adjoint tangential operators]
 For each  $i=1,\dots,n+1$, we define the \emph{adjoint tangential operator} $\de_i^*:C^1(\Sigma)\to C(\Sigma^*)$ 
  through the identity
 \begin{equation} 
 \label{IP}
  \int_\Sigma \psi   \,  \de_i\varphi\, d\s = -\int_\Sigma \varphi\, \de_i^*\psi\, d\s, \quad \varphi,\psi\in C^1_c(\Sigma).
 \end{equation} 
\end{definition}

The explicit formula for adjoint operators is given in the next lemma.

\begin{lemma}
 \label{lem:deltai*} Let $\Sigma\subset\R^{n+1}$ be a hypersurface with $\a$-mean curvature $H_\S$.
 For every $\psi\in C^1_c(\Sigma)$ and $i=1,\ldots,n+1$, we have on $\Sigma^*$
 \begin{equation}
  \label{eq:deltai*}
  \de_i^*\psi=-\de_i\psi-\psi\Big[\de_i\log |x|^\a +\frac{1}{\nu_{n+1}} \Big( \de_{n+1}\nu_i-\de_i\nu_{n+1}\Big) +nH_\S \nu_i\Big].
 \end{equation}
\end{lemma}

\begin{proof}
 Let $\phi,\psi\in C^1_c(\S^*)$.  Then by the area formula 
 \eqref{eq:a_perimeter_graph} with $v$ as in \eqref {eq:per_measure}
 \begin{equation}
  \begin{split}
   \int_\S\phi\, \de_i^*\psi\, d\s&=\int_\S\psi\, \de_i\phi\, d\s=
   \int_\Omega\psi(X_i\phi-\<X\phi,\nu\>\nu_i)vd\x\\
   &=-\int_\Omega\phi\Big[X_i(v\psi)-\sum_{k=1}^nX_k(\nu_i\nu_kv\psi)\Big]d\x\\
   &=-\int_\S\phi\frac{A}{v}d\s,
  \end{split}
 \end{equation}
 where in the last identity we set $A=X_i(v\psi)-\sum_{k=1}^nX_k(\nu_i\nu_kv\psi)$. We are left to prove that $A/v$ is equal to the right-hand side in \eqref{eq:deltai*}.
 
 We have
 \begin{equation}
  \begin{split}
   -\frac{A}{v}&=-\frac{1}{v}\left[\psi X_iv + vX_i\psi-v\nu_i\<X\psi,\nu\>-\psi\sum_{k=1}^nX_k(\nu_i\nu_kv)\right]\\
   &=-\de_i\psi-\frac{\psi}{v}\left[X_iv-\nu_i\<Xv,\nu\>-v\sum_{k=1}^nX_k(\nu_i\nu_k)\right].
  \end{split}
 \end{equation}
 In the term within brackets above, we easily recognize $X_iv-\nu_i\<Xv,\nu\>=\de v$. On the other hand, using $X_{n+1}\nu=0$ and \eqref{sic}, we get
 \begin{equation}
  \sum_{k=1}^n X_k(\nu_i\nu_k)=-n\nu_iH_\S -
  \frac{1}{\nu_{n+1}}\Big[X_{n+1}\nu_i-\nu_{n+1}\<X\nu_i,\nu\>\Big]
  =-n\nu_iH _\S- \frac{1}{\nu_{n+1}}\de_{n+1}\nu_i.
 \end{equation}
 Summarizing, we have
 \begin{equation}
  \frac{A}{v}=\de_i\psi+\psi\left[\frac{\de_i v}{v}+\frac{\de_{n+1}\nu_i}{\nu_{n+1}}+nH_\S \nu_i\right].
 \end{equation}
 To prove our claim it then remains to check the following identity
 \begin{equation}
  \frac{\de_iv}{v}=\frac{\de_ia}{a}-\frac{\de_i\nu_{n+1}}{\nu_{n+1}}=\de_i(\log a)-\frac{\de_i\nu_{n+1}}{\nu_{n+1}}.
 \end{equation}
 Indeed, since $\nu_{n+1}=a/v$ we have
 \begin{equation}
  \frac{\de_ia}{a}-\frac{\de_i\nu_{n+1}}{\nu_{n+1}}=\frac{\de_ia}{a}-\frac{v\de_i(\frac{a}{v})}{a}=\frac{\de_ia}{a}-\frac{v}{a}\left(\frac{\de_ia}{v}-\frac{a}{v^2}\de_iv\right)=\frac{\de_iv}{v}.
 \end{equation}
\end{proof}

\begin{definition}[Tangential Laplacians] Let $\Sigma\subset\R^{n+1}$ be a hypersurface of class $C^2$.
The \emph{tangential Laplacian} of $\Sigma$ is the operator 
$  \cL_\Sigma:C^2(\Sigma)\to C(\S^*)$
\begin{equation}
\label{eq:Delta*_Sigma}
    \cL_\Sigma\varphi=-\sum_{i=1}^{n+1}\de_i^*\de_i\varphi.
\end{equation}
\end{definition}

The relation between  $\cL_\Sigma$  and the non-adjoint Laplacian 
\begin{equation}
\label{eq:Delta_Sigma}
    \Delta_\Sigma\varphi=\sum_{i=1}^{n+1}\de_i^2 \varphi.
\end{equation}
 is described in the   following proposition.

\begin{lemma}
 \label{DeltaL} For any  $\varphi\in C^2(\Sigma)$  we have the identity
 \begin{equation}\label{fif}
  \cL_\Sigma\phi=\Delta_\Sigma\phi + \nu_{n+1}^2\langle \de\phi,\de \log a  \rangle + \de_{n+1}\phi\,\de_{n+1} \log a
 \end{equation}
\end{lemma}

\begin{proof}
 We have
 \begin{equation}
  \cL_\S \phi=-\sum_{i=1}^{n+1}\de_i^*\de_i=\De_\S\phi+\<\de\phi,\de\log a\>-\frac{1}{\nu_{n+1}}\<\de\phi,\de\nu_{n+1}-\de_{n+1}\nu\>.
 \end{equation}
 By formula \eqref{eq:deltai*}, for $i=1,\dots,n+1$ we have
 \begin{equation}
  \de\nu_{n+1}-\de_{n+1}\nu
  =\nu_{n+1}\big[\de\log a-\big(\bar0,\de_{n+1}\log a\big)\big]+
  \nu_{n+1}^2\big[\nu\de_{n+1}\log a-\nu_{n+1}\de\log a\big].
 \end{equation}
 Since $\<\de(\cdot),\nu\>=0$, we deduce
 \begin{equation}
  \begin{split}
  -\frac{1}{\nu_{n+1}}\<\de\phi,\de\nu_{n+1}-\de_{n+1}\nu\>=
  &-\<\de\phi,\de\log a\> + \de_{n+1}\phi\de_{n+1}\log a\\
  &+\nu_{n+1}^2\<\de\phi,\de\log a\>,
 \end{split}
 \end{equation}
 proving the result.
\end{proof}

The formal Hessian of $\varphi$ with respect to the vector-fields $X_1,\ldots,X_{n+1}$ is the $(n+1) \times (n+1) $ matrix
\[
       \X \phi  =\big(X_i X_j\varphi   )_{i,j=1,\ldots,n+1}.
\]
The non-adjoint Laplacian $\Delta_\S$ has a clear representation in terms of the Grushin operator $\cL$ in \eqref{eq:cL}, $\X$ and $\a$-mean curvature of $\S$.

\begin{lemma} \label{AMB}
Let $\Sigma\subset\R^{n+1}$ be a hypersurface of class $C^2$ with $\a$-mean curvature $H_\S$. 
 For any $\phi\in C^2(\R^{n+1})$ we have the identity on $\S^*$ 
 \begin{equation}\label{fuf}
    \De_\S\phi = \cL\phi-\<(\X\phi)\nu,\nu\>+nH_\S \<X\phi,\nu\>.
 \end{equation}
\end{lemma}

\begin{proof}
 We first compute $\de_i^2\phi$, for $i\geq1$. We have
 \begin{equation}
  \begin{split}
   \de_i^2\phi&=\de_i(X_i\phi-\<X\phi,\nu\>\nu_i)\\
   &=X_iX_i\phi-\<XX_i\phi,\nu\>\nu_i - \<X\phi,\nu\>\de_i\nu_i - \de_i(\<X\phi,\nu\>)\nu_i.
  \end{split}
 \end{equation}
Summing over $i$, we obtain the identities
 \begin{align}
  &\sum_{i=1}^{n+1}\<XX_i\phi,\nu\>\nu_i=\<(\X\phi)\nu,\nu\>,\quad 
  \sum_{i=1}^{n+1}\de_i\nu_i=-nH_\S,\\
  &\sum_{i=1}^{n+1}\de_i(\<X\phi,\nu\>)\nu_i=\<\de(\<X\phi,\nu\>),\nu\>=0,
 \end{align}
 and this completes the proof.
\end{proof}

\newcommand{\G}{\Gamma}

\renewcommand{\rho}{\varrho}

We specialize the previous formulas to the case when $\phi$ is a radial function around $0\in\R^{n+1}$. The symmetry is governed by the gauge function $\rho$ in 
\eqref{def:Gauge}.
Below, we collect the differential identities concerning first and second order derivatives of $\rho$. With the notation $\xi=(\x,\y)$ and $\rho=\rho(\xi)$ we have
\begin{align}
 &\nabla_\x \rho=\x|\x|^{2\a}\rho^{-(2\a+1)}\quad \textrm{and} \quad
 \d_\y\rho=(\a+1)\y\rho^{-(2\a+1)}.
\end{align}
Then the squared norm of the $X$-gradient of $\rho$ is  
\begin{equation}\label{square}
 |X\rho|^2=|\nabla_\x\rho|^2+|x|^{2\a}|\d_\y\rho|^2=|\x|^{2\a}\rho^{-2\a}. 
\end{equation}
The second derivatives of $\rho$ are, with $i,j=1,\dots,n$ and denoting by $\varepsilon_{ij}$ the Kronecker symbol, 
\begin{equation}
\label{D2}
\begin{split}
 X_iX_j\rho&=
 |\x|^{2\a}\Big[\e_{ij}+2\a\frac{x_ix_j}{|\x|^2}-(2\a+1)\frac{x_ix_j|\x|^{2\a}}{\rho^{2(\a+1)}}\Big],
 \\
 X_iX_{n+1}\rho
 &= (\a+1)|\x|^{™2\a}x_i\y\rho^{-2\a-1}\Big[\frac{\a}{|\x|^{\a+2}}-(2\a+1)\frac{|\x|^{2\a}}{\rho^{2(\a+1)}}\Big],
 \\
 X_{n+1}X_j\rho&=-(2\a+1)(\a+1)|\x|^{3\a}x_j\y\rho^{-4\a-3},\\
 X_{n+1}^2\rho&=(\a+1)|\x|^{2\a}\rho(x)^{-2\a-1}\Big[1-(2\a+1)(\a+1)\frac{\y^2}{\rho^{2(\a+1)}}\Big].
\end{split}
\end{equation}
From \eqref{D2},  we get the following   formulas for the   Laplacian $\cL$ and for the quadratic form $\<(X^2\rho)\nu,\nu\>$:
\begin{equation} \label{follo}
 \cL\rho=(n+\a)\frac{|X\rho|^2}{\rho},
\end{equation}
and
\begin{equation}
\label{X^2}
 \begin{split}
  \<(\X\rho)\nu,\nu\>=|\bar\nu|^2 &+ 2\a\frac{\<\x,\bar\nu\>^2}{|\x|^2} - (2\a+1)\frac{|\x|^{2\a}\<\x,\bar\nu\>^2}{\rho^{2(\a+1)}}\\
  &+   (\a+1)\<\x,\bar\nu\>^2\nu_{n+1}\y\left(\frac{\a}{|\x|^{\a+2}} - 2(2\a+1)\frac{|\x|^{2\a}\<\x,\bar\nu\>^2}{\rho^{2(\a+1)}}\right)\\
  &+   (\a+1)\nu_{n+1}^2\left(1-(\a+1)(2\a+1)\frac{\y^2}{\rho^{2(\a+1)}}\right).
 \end{split}
\end{equation}

Let $\S\subset\R^{n+1}$ be a hypersurface with $\a$-normal $\nu$ and $\a$-mean curvature $H_\S$.
The structural function  $q_\S:\S^*\to\R$ introduced in \eqref{eq:exactq}  governs the harmonicity of $\S$ at $0\in\S$ and appears in the following formula.

\begin{theorem}[Tangential Laplacian of radial functions] For any $\phi\in C^2(\R^+)$, the function 
 $\psi\in C^2(\R^{n+1}\setminus\{0\})$, $\psi=\phi\circ\rho$,  satisfies the identity
 \begin{equation}\label{tollo}
  \cL_\S\psi=\Big\{ \phi''(\rho) +\phi'(\rho)\frac{n+\a-1}{\rho}\Big\} |\de\rho|^2+q_\S \, \phi'(\rho).
 \end{equation}\end{theorem}

\begin{proof}
In a first step, we prove that formula \eqref{tollo} holds with the following expression for $q_\S$: 
 \begin{equation}
  \begin{split} \label{kiss}
  q_\S =-\frac{n+\a-1}{\rho}|\de\rho|^2&+(n+\a)\frac{|X\rho|^2}{\rho}-\<(X^2 \rho)\nu,\nu\>+nH_\S \<X\rho,\nu\>\\
  &+\a\nu_{n+1}^2\<\de\rho,\de\log|x|\>+\a\de_{n+1}\rho\de_{n+1}\log|x|.
  \end{split}  
 \end{equation}
 The proof combines formula \eqref{fif} of Lemma \ref{DeltaL} and formula \eqref{fuf} in Lemma \ref{AMB}: 
 \begin{align}
  \De_\S\psi&=\cL(\psi\circ\rho) - \<X^2(\psi\circ\rho)\nu,\nu\>+nH_\S \<X(\psi\circ\rho),\nu\>\\
  &=\phi''(\rho)|\de\rho|^2+\phi'(\rho)\big(\cL\rho-\<(X^2 \rho)\nu,\nu\>+nH_\S\<X\rho,\nu\>\big),
 \end{align}
 where the last identity is a simple computation with the chain rule. On the other hand, we have
 \begin{equation}
  \cL_\S\psi=\De_\S\psi+\a\nu_{n+1}^2\<\de\psi,\de(\log|\x|)\>+\a\de_{n+1}\psi\de_{n+1}(\log|\x|).
 \end{equation}
 Since $\de\psi=\phi'(\rho)\de\rho$, we deduce
 \begin{align}
  \cL_\S\psi=\phi''(\rho)|\de\rho|^2+\phi'(\rho)\big[&\cL\rho-\<(X^2\rho)\nu,\nu\>+nH_\S \<X\rho,\nu\>\\
  &+\a\nu_{n+1}^2\<\de\rho,\de(\log|\x|)\>+\a\de_{n+1}\rho\de_{n+1}(\log|\x|)\big]
 \end{align}
 The proof of \eqref{tollo} with $q_\S$ as in \eqref{kiss} 
  is then concluded by adding and subtracting the quantity  $\frac{n+\a-1}{\rho}|\de\rho|^2$ within squared brackets and using
  \eqref{follo}.

  In the next step, we check that $q_\S$ in \eqref{kiss} is as in \eqref{eq:exactq}.
 We start by observing that an elementary  computation gives  
 \begin{equation}
  \nu_{n+1}^2\<\de\rho,\de\log|x|\>+\de_{n+1}\rho\de_{n+1}\log|x|=\frac{|\x|^{2\a}}{\rho^{2\a+1}}\Big(\nu_{n+1}^2-(\a+1)\nu_{n+1}|\x|^{-2-\a}\<\x,\bar\nu\>\y\Big),
 \end{equation}
 and
 \begin{equation}
  -\frac{n+\a-1}{\rho}|\de\rho|^2+(n+\a)\frac{|X\rho|^2}{\rho}=\frac{|X\rho|^2}{\rho}+\frac{n+\a-1}{\rho}\<X\rho,\nu\>^2.
 \end{equation}
 Inserting formulas \eqref{D2}--\eqref{X^2}  into   \eqref{kiss}, we obtain
 \begin{equation}
 \label{kosso}
  \begin{split}
   q_\S =&\frac{n+\a-1}{\rho}\<X\rho,\nu\>^2 + nH_\S \<X\rho,\nu\>
  \\
  &-\frac{|X\rho|^2}{\rho}\Big[2\a\frac{\<\x,\bar\nu\>^2}{|\x|^2}-(2\a+1)\frac{|\x|^{2\a}}{\rho^{2\a+2}}\<\x,\bar\nu\>^2+2\a(\a+1)\<\x,\bar\nu\>\y|\x|^{-2-\a}\nu_{n+1}
  \\
  &\qquad\qquad\;\; -2(\a+1)(2\a+1)\frac{|\x|^{\a}}{\rho^{2\a+2}}\<\x,\bar\nu\>\y\nu_{n+1}
  \\
  &\qquad\qquad\;\; -(\a+1)^2(2\a+1)\y^2\nu_{n+1}^2\frac{1}{\rho^{2\a+2}}\Big].
  \end{split}
 \end{equation}
 Now we observe that
 \begin{equation}
  \<X\rho,\nu\>^2=\frac{|\x|^{2\a}}{\rho^{4\a+2}}\left(|\x|^\a\<\x,\bar\nu\>+(\a+1)\y\nu_{n+1}\right)^2
 \end{equation}
 and so, since $|X\rho|=(|\x|/\rho)^\a$,
 \begin{equation}\label{zorro}
  \<X\rho,\nu\>=\frac{|X\rho|}{\rho^{\a+1}}\left(|\x|^\a\<\x,\bar\nu\>+(\a+1)\y\nu_{n+1}\right).
 \end{equation}
 Replacing last identity into \eqref{kosso} and after some computations that are omitted, we obtain \eqref{eq:exactq}.
\end{proof}

\section{Mean value formulas}
\label{S4}

We are ready to prove Theorems \ref{MVF} and \ref{thm:SMVF}.
Let $\S\subset\R^{n+1}$ be a hypersurface of class $C^2$ with $0\in\S$. We say that:
\begin{itemize}
\item[i)] $\S$ is $\a$-harmonic at $0$ if $q_\S=0$;

\item[ii)] $\S$ is $\a$-subharmonic at $0$ if $q_\S\geq 0$;
\item[iii)] $\S$ is $\a$-superharmonic at $0$ if $q_\S\leq 0$.  
\end{itemize}

%
%
%
%

\begin{proof}[Proof of Theorem \ref{MVF}] For any $\psi \in C^\infty_c(\S)$, by the integration by parts formula \eqref{IP} we have 
\begin{equation}\label{plot}
0 = \int _\S \cL_\S f\, \psi\,  d\sigma = -\int _\S \langle \delta f \,\delta  \psi  \rangle d\sigma =  \int _\S f \,  \cL_\S\psi\,  d\sigma .
\end{equation}
We shall use this identity for functions $\psi$ with radial structure around $0$.
Let  $\chi\in C^\infty(\R^+)$ be a function such that $\chi(r)=1$  for $0<r<1/2$ and $\chi(r)=0$ for $r>1$.
We may also assume that $\chi'\leq 0$.
With the notation $m(r) = r^{n+\a}$, for $0<s<r$ we define the function $\vartheta_s\in C^\infty(\R^+)$
  \begin{equation}
   \label{eq:chir}
   \vartheta_s(\rho)=\frac{\partial }{\partial  s}\left(\frac{1}{m(s)}\chi\Big(\frac{m(\rho)}{m(s)}\Big)\right),\quad \rho>0.
  \end{equation}
  Assume there exists a solution $\phi_s\in C^\infty(\R^+) $ to the differential problem
   \begin{equation}
   \label{eq:difftosolve2}
\left\{
\begin{array}{l}
   \phi_s''(\rho)+\frac{m''(\rho)}{m'(\rho)} \phi_s'(\rho)=\vartheta_s(\rho),\quad \rho>0\\
   \phi_s(\rho)=0,\quad \rho>s .
   \end{array}
   \right.
  \end{equation}
  Then we may consider the function $\psi_s(\xi) = \phi_s(\rho(\xi))$ for $\xi \in \S$.
  By formula \eqref{tollo} with $q_\S=0$ and \eqref{eq:difftosolve2} we have
  \[
  \cL_\S\psi_s=\Big\{ \phi_s''(\rho) +\phi_s'(\rho)\frac{n+\a-1}{\rho}\Big\} |\de\rho|^2=\frac{\partial }{\partial  s}\left(\frac{1}{m(s)}\chi\Big(\frac{m(\rho)}{m(s)}\Big)
\right)  |\de\rho|^2 ,
  \]
  and from \eqref{plot}
  we deduce that
  \[
  0= \int _\S f \,  \cL_\S\psi_s\,  d\sigma = \frac{\partial }{\partial  s} \int _\S f (\xi) \,  \frac{1}{m(s)}\chi\Big(\frac{m(\rho)}{m(s)}\Big)
   |\de\rho|^2\,  d\sigma ,
  \]
  and thus  for any $0<s<r$
  \[
   \frac{1}{m(s)}
  \int _\S f (\xi) \, \chi\Big(\frac{m(\rho)}{m(s)}\Big)
   |\de\rho|^2\,  d\sigma =
    \frac{1}{m(r)}
   \int _\S f (\xi) \, \chi\Big(\frac{m(\rho)}{m(r)}\Big)
   |\de\rho|^2\,  d\sigma .
  \]
  We may approximate the characteristic function of the interval $(0,1)\subset \R$ by a sequence of functions $\chi$ as above.
  Passing to the limit in the previous identity, we get
    \[
   \frac{1}{m(s)}
  \int _{B_s\cap \S} f (\xi) \,  
   |\de\rho|^2\,  d\sigma =
    \frac{1}{m(r)}
   \int _{B_r\cap \S}  f (\xi) \,     |\de\rho|^2\,  d\sigma .
  \]
  This formula holds for $f=1$, proving that the right hand-side of \eqref{mvfC} does not depend on $r>0$.
  By continuity of $f$ at $0$, we get \eqref{mvf} with $C_{\S,n,\a}$ as in \eqref{mvfC}.
  
  We are left to show that problem \eqref{eq:difftosolve2} has a solution.
  A straightforward computation shows that the function $\vartheta_s$ in \eqref{eq:chir} reads
  \begin{equation}
  \label{eq:chir2}
  \begin{split}
   \vartheta_s(\rho)&
   =-\frac{m'(s)}{m(s)^2}
   \frac{\partial}{\partial \rho}\left(m(\rho)\chi\Big(\frac{m(\rho)}{m(s)}\Big)\right),
   \end{split}
  \end{equation}
  and so the differential equation in \eqref{eq:difftosolve2} is equivalent to
 \begin{equation}
 \label{eq:difftosolve3}
 \frac{\partial }{\partial \rho}\big(m'(\rho)\phi_s'(\rho)\big)=-\frac{m'(s)}{m(s)^2}\frac{\partial }{\partial \rho}\left(m(\rho)\chi\Big(\frac{m(\rho)}{m(s)}\Big)\right).
\end{equation}
Integrating with $\phi'_s(\rho)=0$ for $\rho>s$ we obtain:
 \begin{equation}
 \label{eq:difftosolve4}
\phi_s'(\rho) =-\frac{m'(s)m(\rho)}{m(s)^2 m'(\rho)} \chi\Big(\frac{m(\rho)}{m(s)}\Big) =-\frac{\rho}{s^{n+\a+1} } \chi\Big(\frac{\rho^{n+\a} }{s^{n+\a}}\Big) .
\end{equation}
A final integration with $\phi_s(\rho)=0$ for $\rho>s$ yields
\begin{equation}\label{stop}
\phi_s(\rho) = \int_\rho ^\infty \frac{r}{s^{n+\a+1} } \chi\Big(\frac{r^{n+\a} }{s^{n+\a}}\Big)dr,
\end{equation}
showing that we find a function satisfying as a matter of fact $\phi_s(\rho)=0$ for $\rho>s$.

\end{proof}

\begin{rmk} 
Using the technique of Theorem~\ref{MVF}, with $m(r)= r^{n+\a}$ replaced by $m(r)=r^{n+\a+1}$, one obtains a mean value formula for $\cL$-harmonic functions at $0\in\R^{n+1}$, where $|\delta\rho|$ in \eqref{MVF} is replaced by $|X\rho|$, and $\cL$ is the Grushin Laplacian \eqref{eq:cL}.
The same technique works when $\cL=\sum_{i=1}^m X_j^2$ is the sub-Laplacian of any family 
$X_1,\dots, X_m$ of smooth vector fields in $\R^{n+1}$ satisfying the H\"ormander condition,  with $2\leq m\leq n+1$ { and admitting a global fundamental solution}. The resulting mean-value formulas coincide with the formulas  obtained in~\cite{CL21}.

We explain the relation between the two approaches in the case of a Carnot group {  of topological dimension $d>2$}.  Let $\Gamma$ be the fundamental solution of the corresponding Carnot sub-Laplacian  $\cL$ with pole at $0$. For a harmonic function $\cL f=0$, the mean value formula (1.4) proved in~\cite{CL21}
reads
\begin{equation}
\label{q:CL}
f(0)=\frac{1}{r}\int_{\{\xi\in\R^d : \Gamma(\xi)>\frac{1}{r}\}}f(\xi)\left| X (\log \Gamma)\right|^2\varphi\left(\frac{1}{r\Gamma(\xi) }\right)\;d\xi,\quad r>0,
\end{equation}
where $\varphi$ is any continuous function on the interval $[0,1]$ with unit integral.

Let
{ $\varrho(\xi) = \Gamma(\xi)^{\frac{1}{Q-2}}$}, $\xi\neq0$, where $Q\in\N$ is the homogeneous dimension of the group. The Lebesgue measure of the  balls $B_s =\{\xi\in\R^{d}:\varrho(\xi)<s\}$ satisfies $ m(s) =\mathcal L^{d}(B_s) = C s^ Q$ for some constant $C>0$ depending on $n$ and $Q$.
Using the technique of  Theorem~\ref{MVF} we get the mean value formula
\begin{equation}\label{eq:mvf_horm}
f(0) =\frac{C_{d,Q} }{m(s)} \int_{B_s}f(\xi) \, |X\rho(\xi)|^2  d\xi,
\end{equation}
with $C_{d,Q}>0$ fixed on choosing $f=1$. Formula \eqref{eq:mvf_horm}
is precisely formula \eqref{q:CL} with  
 the choice $\varphi(t)=\frac{Q}{Q-2} t^{\frac{2}{Q-2}}$. In fact, in this case we have
\[
\left| X (\log \Gamma)\right|^2\varphi\left(\frac{1}{r\Gamma}\right)
=Q(Q-2) r^{\frac{2}{2-Q}} |X \rho|^2,
\] 
and we obtain equivalence with \eqref{eq:mvf_horm} by setting $r=s^{Q-2}$.

\end{rmk}

\begin{proof}[Proof of Theorem  \ref{thm:SMVF}]
The proof is identical to the proof of Theorem \ref{MVF} with minor modifications that we list below.
For any nonnegative $\psi \in C^\infty_c(\S)$,   we have 
\begin{equation}\label{plot2}
0 \leq     \int _\S f \,  \cL_\S\psi\,  d\sigma .
\end{equation}
The function $\phi_s$ is the solution to \eqref{eq:difftosolve2} defined in \eqref{stop}. Notice that $\phi_s'\leq 0$ if $\chi\geq 0$ in \eqref{eq:difftosolve4}.
Then we have $q_\S (\xi)\phi_s' (\rho(\xi))\leq 0$ for  $\xi\in \S^*$. By formula \eqref{tollo}, the function $\psi_s = \phi_s\circ\rho$ thus satisfies
\[
\cL_\S\psi_s\leq \frac{\partial }{\partial  s}\left(\frac{1}{m(s)}\chi\Big(\frac{m(\rho)}{m(s)}\Big)
\right)  |\de\rho|^2,
\] 
and we get, for $0<s<r$,
\[
   \frac{1}{m(s)}
  \int _{B_s\cap \S} f (\xi) \,  
   |\de\rho|^2\,  d\sigma \leq 
    \frac{1}{m(r)}
   \int _{B_r\cap \S}  f (\xi) \,     |\de\rho|^2\,  d\sigma .
\]
The choice $f=1$ shows the existence of the limit in \eqref{mvfC2}.

\end{proof}

\begin{rmk} If  $\S $ is $\a$-superharmonic, $q_\S\leq 0$, then a function $f$ with $   \cL_\Sigma f\leq 0$
  satisfies the super-mean-value formula at $0$:
 \begin{equation}
  \label{mvf-2}
  f(0) \geq \frac{C_{\S,n,\a} }{r^{n+\a} } \int_{B_r\cap \S }f(\xi) \, |\delta\rho(\xi)|^2  d\sigma.
 \end{equation}
The proof is the same as in the sub-harmonic case.

\end{rmk}

\section{Analysis of the structural function $q_\S$}
\label{S5}

\subsection{Homogeneous hypersurfaces are harmonic}\label{ss:homo}

In $\R^{n+1}$ with the Grushin structure, we  introduce the anisotropic dilations $\dil_\lambda:\R^{n+1}\to\R^{n+1}$, $\lambda>0$,
\[
\dil_\lambda (\xi) =\dil_\lambda (\x,\y) = (\lambda \x,\lambda^{\alpha+1} \y),\quad \xi\in\R^{n+1}.
\]
We say that a set $\S\subset\R^{n+1}$ is $\dil_\lambda$-homogeneous if $\dil_\lambda(\S)=\S$ for any $\lambda>0$.

\begin{lemma} 
Let $\S\subset\R^{n+1}$ be a hypersurface of class $C^2$ with $0\in\S$.
If $\S$ is $\dil_\lambda$-homogeneous then $\S$ is $\a$-harmonic at $0$.
\end{lemma}

\begin{proof} We check the claim when $\S$ is a $\y$-graph
$
  \S =\{(\x,u(\x))\in\R^{n+1}: x\in\R^n\}$ for some function  $u\in\ C^2(\R^n)$ satisfying  the identity $u(\lambda \x) =\lambda^{\a+1} u(\x)$ for any $\lambda>0$ and $\x\in\R^n$.
Differentiating this identity at $\lambda =1$ we get
\begin{equation}\label{AA}
 \langle \nabla u(\x),\x\rangle = (\alpha+1) u(\x),\quad \x\in \R^n.
 \end{equation}
 Using formulas \eqref{eq:nu_graph} for the $\a$-normal $\nu =(\bar\nu,\nu_{n+1})$, \eqref{AA} is equivalent to
 \[
 |\x|^\a  \langle \x,\bar\nu\rangle + (\alpha+1) \y \nu_{n+1}=0,\quad (\x,\y)\in\S^*. 
 \]
 By formula \eqref{zorro} this is in turn equivalent to $\langle X\rho,\nu\rangle =0$ on $\S^*$, and this implies   $q_\S=0$, see \eqref{eq:exactq}.
\end{proof}
 
When $\S$ is $\dil_\lambda$-homogeneous we have $\langle X\rho,\nu\rangle=0$ and, using \eqref{square},  the kernel $|\delta\rho|^2$ appearing in the mean value formula \eqref{mvfC} reduces to
\[
|\delta\rho|^2= |X\rho|^2 =\frac{|\x|^{2\a}}{\rho^{2\a}}.
\]
This kernel is $0$-homogeneous with respect to the dilations $\dil_\lambda$ and satisfies $|\delta\rho|^2\leq 1$.

\subsection{Flat case}
The hyperplane $\S= \{(x,y)\in\R^{n+1}: y=0\}$ is $\dil_\lambda$ homogeneous and it is therefore $\a$-harmonic. 
The $\a$-normal is constant, $\nu =(0,1)\in\R^n\times\R$, and it follows that $\S$ is  also $\a$-minimal, $H_\S=0$.

The tangential gradient $\delta$
reduces to the standard gradient $\delta\varphi = (\nabla_x\varphi,0)$ and the non-adjoint Laplacian $\Delta_\S$ in \eqref{eq:Delta_Sigma} reduces to the standard Laplacian $\Delta_x$ in $\R^n$, see \eqref{fuf}. From formula \eqref{fif}
we deduce that
\[
\cL_\S \varphi =  \Delta_x\varphi +\a  \langle \nabla_x\log |x|,\nabla_x\varphi\rangle = \frac{1}{|\x|^\a} \mathrm{div}\Big( |\x|^\a \nabla_x \varphi \Big).
\]
Theorem \ref{MVF} states that a function $\varphi \in C^2(\R^n)$ satisfying $\cL_\S \varphi =0$ has the mean value property at $0$
\[
  \varphi (0) = \frac{n+\a }{n\omega_n r^{n+\a}}\int_{\{|\x|<r\} } \varphi (x) |\x|^\a  dx,
\]
with $\omega_n = \mathcal L^n(\{|x|<1\})$.

\subsection{$\a$-subharmonic surfaces}
In this section, we look for sufficient conditions for a hypersurface $\S$ to be $\a$-subharmonic at $0\in\S$.

\begin{definition} [$\eta$-flatness] Let $\eta>0$. 
We say that a hypersurface $\S\subset\R^{n+1}$ is $\eta$-flat  at $0\in\S$ if there exists $r>0$ such that its $\a$-normal $\nu =(\bar\nu,\nu_{n+1})$ satisfies
\begin{equation} \label{eta-flat}
(\alpha+1)  |\y| |\nu_{n+1}| \leq \eta |\x|^\a  |\langle  x,\bar\nu \rangle |
\end{equation}
for all points $\xi=(\x,\y) \in \S^*\cap B_r$.
\end{definition}

When $\S$ is the $\y$-graph of a function $u$, condition \eqref{eta-flat} reads
 \begin{equation}
  \label{eq:stella}
  (\a+1)|u|\leq \eta\big|\< \x,\nabla_\x u \>\big|
 \end{equation}
holding for points a  neighborhood of $0\in\R^n$.

\begin{lemma}
 \label{lem:stellato2stella} Let $u\in C^1( \{|\x|\leq 1\})$ be a function satisfying \eqref{eq:stella}.
 Then  for any point  $|\x|\leq 1$ we have
 \begin{equation}
  \label{eq:2stelle}
  |u(\x)|\leq\Big(\max_{|\x|=1}|u|\Big)|\x|^{\frac{\a+1}{\eta}}.
 \end{equation}

\end{lemma}

\begin{proof} Let $|\x|=1$ be fixed. We prove the claim along the segment $t\x$, with $t\in[0,1]$.  
Letting $\phi:[0,1] \to\R$, $\phi(t)=u(t\x)$, assumption \eqref{eq:stella} reads
 \begin{equation}
  \label{eq:stella'}
  (\a+1)|\phi(t)|\leq\eta t|\phi'(t)|.
 \end{equation}
 If $\phi=0$ on  $[0,1]$, the claim is trivial. Then we can assume that the open set  $A=\{t\in(0,1):\phi(t)\neq 0\}$ is nonempty.
 This set is a finite or countable disjoint union of intervals $(a,b)\subset A$. We always have  $\phi(a)=0$ and from
  \eqref{eq:stella'} it follows that $\phi'\neq 0$ on $(a,b)$, say $\phi'(t)>0$ for any $t\in(a,b)$.
  Then $\phi$ is strictly monotone increasing and thus $\phi(b)>0$, and so $b=1$. It follows that  $A=(a,1)$, for some $a\in [0,1)$, and $\phi=0$ on $[0,a]$.

 We can without loss of generality assume that $a=0$ and conclude the proof in the following way. We have $\phi>0$ and, say, $\phi'>0$ on  $(0,1)$. Then \eqref{eq:stella'} reads
 \begin{equation}
  \frac{d}{dt}\log(t^{\a+1})\leq \eta\frac{d}{dt}\log\phi(t),
 \end{equation}
 and integrating this inequality  from $t=s$ to $t=1$,   $0<s<1$, we get
 $  \phi(s)\leq\phi(1)s^{\frac{\a+1}{\eta}}$. This implies  \eqref{eq:2stelle} and the proof is concluded. 
  \end{proof}

\begin{theorem}
 \label{thm:q>0}
 Let $\S\subset\R^{n+1}$ by a hypersurface of class $C^2$ with $\a$-normal $\nu = (\bar \nu,\nu_{n+1})$ and $\a$-mean curvature $H_\S$.
 Assume that:
 \begin{itemize}
 \item[i)] $\S$ is $\eta$-flat at $0\in\S$ for some
 \begin{equation}\label{ET}
  0<\eta<\frac{n+\a}{n+3\a}. 
 \end{equation}
 
 \item[ii)] We have, with limit restricted to $\xi=(\x,\y)\in\S$,
 \begin{equation}\label{spec}
 \lim_{\xi\to0} \frac{|\x|^2 H_\S (\xi )}{ \langle \bar \nu(\xi) , x\rangle }=0.
 \end{equation}
 
 \end{itemize}

 Then there exists a $\delta>0$ such that $q_\S\geq 0$ on $\S^*\cap B_\delta$.

\end{theorem}

\begin{proof}  Without loss of generality we may assume that $\S = \{(x,u(x))\in\R^{n+1}:|\x|\leq 1\}$ for some function $u\in C^2(\{|\x|\leq 1\})$,
and that \eqref{eta-flat}
 holds with $r=1$. Letting $C_1=\max_{|\x|=1}|u(x)|$, points $(x,y)\in\S$ satisfy by \eqref{eq:2stelle}
\begin{equation}
\label{cillo}
|y| \leq C_1 |x| ^{\frac{\a+1}{\eta}}.
\end{equation}

Formula \eqref{eq:exactq} for $q_\S$ reads
\[
  q_\S(\xi) =\<X\rho,\nu\>\Big[\frac{n+3\a}{\varrho} \<X\rho,\nu\> -\frac{2\a}{|\x|^2} \<\x ,\bar\nu\> +nH_\S\Big],
  \]
and the inequality $q_\S\geq 0$ is thus implied   by
\begin{equation}\label{BIG}
\Big| \frac{2\a}{|\x|^2} \<\x ,\bar\nu\> -nH_\S\Big| \leq \frac{n+3\a}{\varrho} |\<X\rho,\nu\> |.
\end{equation}
Inserting the identity \eqref{zorro} into \eqref{BIG}  and then using \eqref{square}, 
we see that $q_\S\geq0$ is implied by
\begin{equation}\label{BIG2}
\Big| \frac{2\a}{|\x|^2} \<\x ,\bar\nu\> -nH_\S\Big|  
\leq 
  \frac{n+3\a}{\rho^{2\a+2}}|\x|^\a\big||\x|^\a\<\x,\bar\nu\>+(\a+1)y\nu_{n+1}\big|.
\end{equation}
By the $\eta$-flatness condition \eqref{eta-flat}, \eqref{BIG2} is in turn implied by
\begin{equation}\label{BIG3}
\Big| \frac{2\a}{|\x|^2} \<\x ,\bar\nu\> -nH_\S\Big|  
\leq  (1-\eta)
  \frac{n+3\a}{\rho^{2\a+2}}|\x|^{2\a} |\<\x,\bar\nu\>|,
\end{equation} 
 and finally, using assumption \eqref{spec}, \eqref{BIG3} is implied by
 \begin{equation}\label{BIG4}
\Big(2\a+\frac{n|\x|^2 |H_\S| }{|\langle x,\bar\nu\rangle| } \Big )
\leq  (1-\eta)(n+3\a )
 \frac{|\x|^{2\a+2} }{\rho^{2\a+2}}.
\end{equation} 
 Now observe that, by \eqref{cillo},
  \begin{equation}
  \label{boh}
 \frac{\rho^{2(\a+1)}}{|\x|^{2(\a+1)} }=1+\frac{(\a+1)^2\y^2}{|\x|^{2(\a+1)}}\leq 1+C_1 (\a+1)^2 |\x|^{2(\a+1)(1/\eta-1)}.
 \end{equation}
 By \eqref{boh} and \eqref{spec},
  inequality \eqref{BIG4} is satisfied for $x=0$ as a strict inequality, by \eqref{ET}.
   Our claim follows by a limiting argument.
 
\end{proof}

\bibliographystyle{abbrv}
\bibliography{biblio}

\end{document}